\documentclass{amsart}

\usepackage{amssymb,amsmath,amsthm}
\usepackage{mathrsfs}
\usepackage{mathtools}
\usepackage[pagebackref]{hyperref}
\usepackage{color}
\usepackage{verbatim}
\mathtoolsset{showonlyrefs}

\renewcommand*{\backref}[1]{}
\renewcommand*{\backrefalt}[4]{%
  \ifcase #1 (Not cited.)%
  \or        (Page~#2.)%
  \else      (Pages~#2.)%
  \fi}

\renewcommand{\[}{\begin{equation}\begin{aligned}}
\renewcommand{\]}{\end{aligned} \end{equation}}

\newcommand{\ddb}{\sqrt{-1}\partial\bar\partial}
\newcommand{\ddbar}{\partial\bar\partial}
\newcommand{\bddbar}{\sqrt{-1}\partial_b\bar{\partial_b}}
\newcommand{\bddb}{\partial_b\bar{\partial_b}}

\newtheorem{thm}{Theorem}
\newtheorem{prop}[thm]{Proposition}
\newtheorem{lemma}[thm]{Lemma}

\newtheorem{cor}[thm]{Corollary}
\newtheorem{conj}[thm]{Conjecture}

\theoremstyle{remark}
\newtheorem{remark}[thm]{Remark}

\theoremstyle{definition}

\theoremstyle{definition}
\newtheorem{exmp}{Example}[section]

\numberwithin{equation}{section}
\numberwithin{thm}{section}

\author{Shih-Kai Chiu}
\address{Mathematical Institute, University of Oxford, Oxford, UK}
\email{Shih-Kai.Chiu@maths.ox.ac.uk}

\title[Nonuniqueness of Calabi-Yau metrics]{Nonuniqueness of
  Calabi-Yau metrics with maximal volume growth}

\date{}

\begin{document}

\begin{abstract}
  We construct a family of inequivalent Calabi-Yau metrics on
  $\mathbf{C}^3$ asymptotic to $\mathbf{C} \times A_2$ at infinity, in
  the sense that any two of these metrics cannot be related by a
  scaling and a biholomorphism. This provides the first example of
  families of Calabi-Yau metrics asymptotic to a fixed tangent cone at
  infinity, while keeping the underlying complex structure fixed. We
  propose a refinement of a conjecture of Sz\'ekelyhidi~\cite{Sz20}
  addressing the classification of such metrics.
\end{abstract}

\maketitle

\section{Introduction}

Since the celebrated work of Yau~\cite{Yau}, Calabi-Yau manifolds have
been studied intensively in K\"ahler geometry, complex algebraic
geometry and physics. In the complete non-compact case, much has been
known in $2$ complex dimensions since the foundational works of
Kronheinmer~\cite{K1}\cite{K2} (see for example
\cite{CC15}\cite{CC19}\cite{CC21}\cite{SZ} and the references
therein). In higher dimensions, Conlon-Hein~\cite{CH} recently
classified asymptotically conical Calabi-Yau manifolds, building on
the important work of Tian-Yau~\cite{TY}.

In this paper, we are interested in Calabi-Yau manifolds with maximal
volume growth, which include asymptotically conical manifolds. In this
more general setting, the tangent cones at infinity are still
Calabi-Yau cones. However, in general these cones can have
non-isolated singularities. Many examples of Calabi-Yau manifolds with
maximal volume growth and singular tangent cones at infinity have been
constructed over the years. Biquard-Gauduchon~\cite{BG} constructed
hyperk\"ahler metrics on cotangent bundles of certain hermitian
symmetric spaces, whose tangent cones are realized as nilpotent orbit
closures in $sl(N, \mathbf{C})$. Joyce~\cite{Joyce} constructed QALE
metrics as resolutions of $\mathbf{C}^n/ \Gamma$, where the action of
the discrete group $\Gamma$ is not free. This approach has been
generalized by Conlon-Degeratu-Rochon~\cite{CDR} to admit more
complicated singularities. More recently, Conlon-Rochon~\cite{CR},
Li~\cite{Li} and Sz\'ekelyhidi~\cite{Sz19} constructed Calabi-Yau
metrics on $\mathbf{C}^3$ with tangent cone given by
$\mathbf{C} \times A_1$ at infinity. Here $A_1$ denotes the singular
hypersurface given by
$ \{x_1^2 + x_2^2 + x_3^2 = 0 \} \subset \mathbf{C}^3$ equipped with
the flat cone metric. We remark that in \cite{CR} and \cite{Sz19},
there are various generalizations in higher dimensions that admit
tangent cones of the form $\mathbf{C} \times V$, where $V$ is a
Calabi-Yau cone with an isolated singularity at the vertex.

The classification of Calabi-Yau manifolds with maximal volume growth
is still largely an uncharted territory. To begin, it is expected that
the tangent cones at infinity are unique, as they are affine varieties
\cite{LSz}. Therefore one might to try to classify Calabi-Yau
manifolds asymptotic to a certain tangent cone at infinity. A recent
breakthrough that fits into this picture is due to
Sz\'ekelyhidi~\cite{Sz20}, who showed that the Calabi-Yau metric on
$\mathbf{C}^n$ asymptotic to $\mathbf{C} \times A_1$ is unique up to
scaling and biholomorphism. Their method is to compare the unknown
metric to scalings of a model metric using better and better
holomorphic gauges. These gauges are given by adapted sequences of
bases in Donaldson-Sun theory~\cite{DS17} in combination with certain
automorphisms of the cone at infinity. The next simplest case is to
study Calabi-Yau metrics on $\mathbf{C}^3$ asymptotic to
$\mathbf{C} \times A_2$ at infinity, where $A_2$ is the singular
hypersurface given by
$\{ x_1^2 + x_2^2 + x_3^3 = 0 \} \subset \mathbf{C}^3$. An example of
such a metric has been obtained by Sz\'ekelyhidi in \cite{Sz19}.

To state our result, we recall the following setup originally
considered in \cite{Sz19}. Consider the hypersurface
$X_1 \subset \mathbf{C}^{n+1}$ given by the equation
\[
z + f(x_1, \ldots, x_n) = 0,
\]
where $f: \mathbf{C}^n \to \mathbf{C}$ is a polynomial, so $X_1$ is
biholomorphic to $\mathbf{C}^n$. Write
$\mathbf{x} = (x_1,\ldots, x_n)$.

\theoremstyle{definition}
\newtheorem{setup}[thm]{Setup}
\begin{setup}
  \label{setup}
  We impose the following restrictions
  on $f$:
  \begin{itemize}
  \item $x_i$ has weight $w_i > 0$ under the action of $t \in \mathbf{C}^*$:
    \[
      t \cdot x_i = t^{w_i}x_i.
    \]
  \item $f$ is homogeneous of degree $d > 1$:
    \[
      t \cdot f(\mathbf{x}) = f(t \cdot \mathbf{x}) = t^d f(\mathbf{x}).
    \]
  \item $V_0 = f^{-1}(0) \subset \mathbf{C}^n$ has an isolated singularity at
    $0 \in \mathbf{C}^n$.
  \item $V_0$ admits a Calabi-Yau cone metric $\omega_{V_0}$
    compatible with the $\mathbf{C}^*$ action.
  \end{itemize}

\end{setup}

Suppose that we are in the above setup. Let
$V_1 = \{ 1+ f(\mathbf{x}) = 0 \} \subset \mathbf{C}^n$. Then $V_1$
admits by \cite{CH13} a unique asymptotically conical Calabi-Yau
metric $\omega_{V_1}$ with asymptotic cone $V_0$ (see
Section~\ref{sec:x1} for the precise meaning of uniqueness).

We would like to degenerate $X_1$ to its ``tangent cone at infinity'':
let us define a $\mathbf{C}^*$ action on $\mathbf{C}^{n+1}$ given by
$F_t(z,\mathbf{x}) = (tz, t\cdot \mathbf{x})$. Then $F_t^{-1}X_1$ has
the equation
\[
t^{1-d}z + f(\mathbf{x}) = 0.
\]
Since $d > 1$, as $t \to \infty$, $F_t^{-1}X_1 \to X_0$, where
\[
  X_0 = \mathbf{C} \times V_0
\]
is equipped with the Calabi-Yau cone metric
$\omega_0 = \ddb |z|^2 + \omega_{V_0}$. This fits into the framework
of Donaldson-Sun theory~\cite{DS17} (see also \cite{Liu} for the case
when the tangent cone at infinity is smooth but the manifold is not
necessarily polarized). In \cite{Sz19}, Sz\'ekelyhidi constructed a
Calabi-Yau metric on $X_1$ asymptotic to $X_0$ at infinity. From the
fibration point of view, the map $z: X_1 \to \mathbf{C}$ has regular
fibers biholomorphic to $V_1$, and the central fiber is given by
$V_0$. Roughly speaking, the metric on $X_1$ can be seen as a
perturbation of the ``semi-Ricci-flat'' metric which restricts to
scalings of $\omega_{V_1}$ on the regular fibers and $\omega_{V_0}$ on
the central fiber.

In this paper, we restrict to the case when $n=3$. Set
$f = x_1^2 + x_2^2 + y^3$, where we write $y = x_3$, so $V_0$ is the
$A_2$ singularity. Recall that $V_0 \cong \mathbf{C}^3/\mathbf{Z}_3$
is equipped with the flat cone metric. The variables $z, x_1, x_2, y$
have weights $1,3,3,2$, respectively, and so $d = 6$ (see
Example~\ref{exmp:a2} for more details). We consider the hypersurface
$X_{1,b} \subset \mathbf{C}^4$ given by
\[
  z + b y + x_1^2 + x_2^2 + y^3 = 0,
\]
where $b \in \mathbf{C}$. Under the $\mathbf{C}^*$ action $F_t$,
$X_{1,b}$ still degenerates to $X_0$. However, the fibration structure
is different from $X_1$ considered in \cite{Sz19} when $b \ne 0$:
there are now two singular fibers, each of which has one $A_1$
singularity. For each $b \in \mathbf{C}$, we construct Calabi-Yau
metrics on $X_{1,b}$ asymptotic to $X_0$. We then distinguish these
metrics using certain normalization of holomorphic functions with
polynomial growth. As a consequence, we obtain the main theorem of
this paper:

\begin{thm}
\label{thm:A2}
There exists a family of Calabi-Yau metrics $\omega_b$,
$b \in [0,\infty)$, on $\mathbf{C}^3$ with tangent cone
$\mathbf{C} \times A_2$ at infinity. Any $\omega_b$ and $\omega_{b'}$
are related by a biholomorphism and a scaling if and only if $b = b'$.
\end{thm}

One way to understand this phenomenon of nonuniqueness is that these
metrics should correspond to different ways to smooth out the $A_2$
singularity. In particular, each $X_{1,b}$ has a distinct fibration
structure, with distinct singular fiber positions and singularity
types.

In Sections \ref{sec:x1} and \ref{sec:x1b}, we describe our
construction of $\omega_b$ by a gluing technique similar to the one
used in \cite{Sz19}. The main difference in our case is that the
fibration is more complicated, and as a result the approximate
solution is not obvious to write down. A crucial observation is that
in our case, away from the singular fibers and the origin, the metric
should still be modeled on either $\mathbf{C} \times V_0$ or
$\mathbf{C} \times V_1$ depending on the regions. This allows us to
write down an approximate solution on $X_{1,b}$ using the approximate
solution on $X_1$ and the nearest point projection from $X_1$ to
$X_{1,b}$ (outside large compact sets) with respect to a certain cone
metric on the ambient $\mathbf{C}^4$.

In Section~\ref{sec:ds}, we describe our method for distinguishing
these metrics, and conclude the proof of Theorem~\ref{thm:A2}. In
particular, we generalize the application of Donaldson-Sun
theory~\cite{DS17} as seen in \cite{Sz20} to construct special
embeddings of Calabi-Yau metrics on $\mathbf{C}^3$ with tangent cone
$\mathbf{C} \times A_2$ at infinity. We also obtain a normalization of
holomorphic functions from the gluing construction in the previous
sections. Our method for distinguishing these metrics is then a
combination these results. At the end of this paper, we propose a
refinement of a conjecture of Sz\'ekelyhidi~\cite{Sz20}, and discuss
preliminary results as well as some difficulties that arise in this
setting. \newline

\noindent{\bf Acknowledgments.} I would like to thank G\'abor
Sz\'ekelyhidi for the encouragement and constant support over the
years. Thanks also to Lorenzo Foscolo and Yang Li for helpful
discussions. I was supported by Simons Collaboration on Special
Holonomy in Geometry, Analysis, and Physics (\#724071 Jason Lotay).

\section{Weighted analysis on $X_1$}\label{sec:x1}

In this section, we explain mostly without proofs the construction of
the approximate solution on $X_1$, as well as the weighted analysis in
\cite{Sz19}. We will however give a detailed proof of
Proposition~\ref{tangentconeisx0} below, since a consequence of its
proof is a normalization of the holomorphic functions with respect to
the approximate metric (see Corollary~\ref{cor:asymptotics}). This
will be used in Section~\ref{sec:ds}.

\subsection{The approximate solution}

We work in Setup~\ref{setup}. Recall in \cite{Sz19} that there is a
cone metric $\ddb R^2$ on $\mathbf{C}^n$, compatible with the
$\mathbf{C}^*$ action, such that the radial function $R$, when
restricting to $V_0$, is uniformly equivalent to the distance function
$r$ on $V_0$. Using $\ddb R^2$, we can extend $r$ homogeneously to a
function, also called $r$, on $\mathbf{C}^n$. $\ddb r^2$ defines a
K\"ahler metric on $V_1$ (away from a large compact set) which is
asymptotic to the Calabi-Yau cone $V_0$ under the nearest point
projection. By \cite[Theorem~2.4]{CH13} and \cite[Theorem~3.1]{CH13},
there exists a unique complete Calabi-Yau metric $\ddb \phi$ on $V_1$
asymptotic to $\ddb r^2$. In particular, $(V_1, \ddb \phi)$ is
asymptotically conical with cone $V_0$.

On $\mathbf{C}^{n+1}$, define $\rho^2 = |z|^2 + R^2$. This gives a
cone metric on $\mathbf{C}^{n+1}$ compatible with the $\mathbf{C}^*$
action.

Let $\gamma_1(s)$ be a cutoff function satisfying
\begin{align*}
\gamma_1(s) =
\begin{cases}
1 & \text{if }s > 2 \\
0 & \text{if }s < 1.
\end{cases}
\end{align*}
and let $\gamma_2 = 1-\gamma_1$. Define the approximate solution, at
least for $\rho > P$ for sufficiently large $P > 0$, by
\begin{align*}
\omega = \ddbar\left(|z|^2 + \gamma_1(R\rho^{-\alpha})r^2 +
\gamma_2(R\rho^{-\alpha})|z|^{2/d}\phi(z^{-1/d}\cdot \mathbf{x})\right),
\end{align*}
where $\alpha \in (1/d, 1)$ is to be chosen later. Writing
$\psi = \phi-r^2$, we can rewrite $\omega$ as
\begin{align*}
\omega = \ddbar\left( |z|^2 + r^2 +
\gamma_2(R\rho^{-\alpha})|z|^{2/d}\psi(z^{-1/d} \cdot \mathbf{x}) \right).
\end{align*}
So the potential of $\omega$ grows like $\rho^2$. In particular if
$\omega$ is positive definite on $\rho > P$, then we can replace
$\omega$ by a metric on $X_1$ that agrees with $\omega$ on
$\rho > 2P$.

The following shows that for large enough $P$, $\omega$ defines a
K\"ahler metric, and the Ricci potential has good enough decay.

\begin{prop}
\label{prop:adecay}
Fix $\alpha \in (1/d,1)$. The form $\omega$ defines a K\"ahler metric
on the subset of $X_1$ where $\rho > P$, for sufficiently large
$P$. For suitable constants $\kappa, C_i >0$ and weight
$\delta < 2/d$, the Ricci potential $h$ of $\omega$ satisfies, for
large $\rho$,
\begin{align*}
|\nabla^i h|_\omega <
\begin{cases}
C_i \rho^{\delta-2-i} & \text{if } R > \kappa\rho \\
C_i \rho^{\delta}R^{-2-i} & \text{if } R \in (\kappa^{-1}\rho^{1/d}, \kappa \rho) \\
C_i \rho^{\delta-2/d-i/d} & \text{if } R < \kappa^{-1}\rho^{1/d}.
\end{cases}
\end{align*}
If in addition $d > 3$ and $\alpha$ is chosen close to $1$, then we
can even choose $\delta < 0$, i.e. in this case $h$ decays faster than
quadratically away from the singular rays.
\end{prop}

Since $\omega$ defines a K\"ahler metric on $X_1 \cap \{ \rho > P \}$,
one can modify the K\"ahler potential so that the new metric is
defined on $X_1$ and coincides with $\omega$ on
$X_1 \cap \{ \rho > 2P \}$, say. This can be done for example using
the ``regularized maximum'' as described in \cite[p.2659]{Sz19}. We
fix a modification of $\omega$ and still call it $\omega$ in the
following.

\subsection{Weighted spaces and tangent cones}

We turn to the definition of weighted spaces. The definition will
account for model geometries in different regions on $X_1$, as
illustrated in the previous proposition. Recall that we want to
perturb the approximate solution $\omega$ to a Calabi-Yau metric on
the set $\{ \rho > A \}$ for sufficiently large $A$. To proceed, we
fix a large $P < A$ such that on $\{ \rho < 2P \}$ we use the usual
$C^{k,\alpha}$ norm. When $\rho > P$ we define the weighted spaces in
terms of the radial distance $\rho$ and the distance to the singular
rays $R$. Define the smooth function
\begin{align*}
w = 
\begin{cases}
1 & \text{if } R > 2\kappa\rho \\
R/(\kappa\rho) & \text{if } R \in (\kappa^{-1}\rho^{1/d}, \kappa\rho) \\
\kappa^{-2}\rho^{1/d-1} & \text{if } R < \frac{1}{2}\kappa^{-1}\rho^{1/d}
\end{cases}
\end{align*}
The three regions in the definition are ``away from singular rays'',
``gluing region'' and ``near singular rays'' in order. Define the
H\"older seminorm as
\[
[T]_{0,\gamma} =
\sup_{\rho(z) > P} \rho(z)^\gamma w(z)^\gamma
\sup_{z' \ne z, z' \in B(z,c)} \frac{|T(z) - T(z')|}{d(z,z')^\gamma}.
\]
Here $c$ is chosen so that $B(z,c)$ has bounded geometry and is
geodesically convex. We use parallel transport along a geodesic to
compare $T(z)$ and $T(z')$. We can now define the weighted spaces
\begin{align*}
\|f\|_{C^{k,\alpha}_{\delta, \tau}} = &\|f\|_{C^{k,\alpha}(\rho < 2P)}
+ \sum_{j=0}^k \sup_{\rho > P} \rho^{-\delta+j}w^{-\tau+j}|\nabla^j f| \\
&+[\rho^{-\delta+k}w^{-\tau+k}\nabla^k f]_{0,\alpha}.
\end{align*}
Alternatively, if we replace $\rho$ by a smoothing of
$\max\{1,\rho\}$, then we can express these weighted norms with
respect to the metric $\rho^{-2}w^{-2}\omega$:
\[
\|f\|_{C^{k,\alpha}_{\delta,\tau}} =
\|\rho^{-\delta}w^{-\tau} f\|_{C^{k,\alpha}_{\rho^{-2}w^{-2}\omega}}.
\]
Using these norms we can define
$C^{k,\alpha}_{\delta,\tau}(X_1, \omega)$. Since we will invert the
Laplacian only on $\rho \ge A$ for $A$ sufficiently large, for $f$
defined on $\rho \ge A$ we define the norms
\[
\|f\|_{C^{k,\alpha}_{\delta,\tau}(\rho^{-1}[A,\infty))} =
\inf_{\hat{f}} \|f\|_{C^{k,\alpha}_{\delta,\tau}(X_1,\omega)},
\]
where the infimum is among all extensions $\hat{f}$ of $f$ on $X_1$.

We record without proof some basic properties of the weighted norms:

\begin{prop}
\label{prop:weightedprop}
The weighted norms we just defined enjoy the following properties:
\begin{itemize}
\item If $f \in C^{k,\alpha}_{a,b}$ and $g \in C^{k,\alpha}_{c,d}$,
  then
  $\|fg\|_{C^{k,\alpha}_{a+c,b+d}} \le \|f\|_{C^{k,\alpha}_{a,b}} \le
  \|g\|_{C^{k,\alpha}_{c,d}}.$
\item If $a < c$, then
$\|f\|_{C^{k,\alpha}_{a,b}} \ge \|f\|_{C^{k,\alpha}_{c,b}}$, and
consequently $C^{k,\alpha}_{a,b} \subset C^{k,\alpha}_{c,b}$. This is
because $\rho > P > 1$.
\item If $b < d$, then
$\|f\|_{C^{k,\alpha}_{a,b}} \le \|f\|_{C^{k,\alpha}_{c,d}}$, and
consequently $C^{k,\alpha}_{a,b} \supset C^{k,\alpha}_{a,d}$. This is
because $w \le 1$.
\end{itemize}
\end{prop}

We can now use the weighted spaces to compare the geometry of $X_1$
with model spaces in different regions. Write $g, g_0$ for the
Riemannian metrics of $\omega, \omega_0$, respectively (recall that
$\omega_0$ is the cone metric on $X_0$). First we consider the region
\[
\mathcal{U} = \{ \rho > A, R> \Lambda\rho^{1/d}\} \cap X_1,
\]
for large $A, \Lambda$, and let
\[
G: \mathcal{U} \to X_0
\]
be the nearest point projection with respect to the cone metric
$\ddbar (|z|^2 + R^2)$ on $\mathbf{C}^{n+1}$. Note that we have
\[
G(z,x) = (z,x')
\]
where $x'$ is the nearest point projection of $x \in \mathbf{C}^n$ with
respect to the cone metric $\ddbar R^2$ on $\mathbf{C}^n$.

\begin{prop}
\label{awaysingular}
Given any $\epsilon > 0$ we can choose $\Lambda > \Lambda(\epsilon)$,
and $A > A(\epsilon)$ sufficiently large so that on $\mathcal{U}$ we have
\[
|\nabla^i(G^*g_0 - g)|_g < \epsilon w^{-i}\rho^{-i}.
\]
for $i\le k+1$. In particular, in terms of weighted spaces we have
\[
\|G^*g_0 - g\|_{C^{k,\alpha}_{0,0}} < \epsilon.
\]
\end{prop}

Next we consider the region where $\rho > A$ but
$R < \Lambda \rho^{1/d}$, i.e. we are close to the singular ray. Fix
$z_0 \in \mathbf{C}$ and a large constant $B>0$. Define
\[
\mathcal{V} = \{ |z-z_0| < B|z_0|^{1/d},\: R < \Lambda \rho^{1/d},\: \rho > A \} \cap X_1.
\]
We will use regions in the form of $\mathcal{V}$ to cover the neighborhood of
the singular ray. We change the coordinates as follows:
\begin{equation}
\hat{x} = z_0^{-1/d} \cdot x, \:\:\: \hat{z} = z_0^{-1/d}(z-z_0).
\end{equation}

Define $\hat{R} = |z_0|^{-1/d}R$, and let
$\hat{\zeta} = \max\{1,\hat{R}\}$. Then $(\hat{z},\hat{x})$ satisfies the equation
\[
z_0^{1/d-1}\hat{z} + 1 + f(\hat{x}) = 0,
\]
and $|\hat{z}| < B, |\hat{R}| < C\Lambda$ for some fixed constant $C$
(since $|z|\sim \rho$). In terms of the new coordinates, we define the
map
\[
H: \mathcal{V} \to \mathbf{C} \times V_1
\]
by $H(\hat{z},\hat{x}) = (\hat{z}, \hat{x}')$, where $\hat{x}'$ is the
nearest point projection of $\hat{x}$ onto $V_1$ with respect to the
ambient cone metric.

\begin{prop}
\label{nearsingular}
Given $\epsilon, \Lambda > 0$, if $A > A(\epsilon,\Lambda, B)$, then we
have
\[
|\nabla^i(H^*g_{\mathbf{C}\times V_1}-|z_0|^{-2/d}g)|_{|z_0|^{-2/d}g} < \epsilon \hat{\zeta}^{-i}
\]
for $i \le k+1$. In terms of weighted spaces we have
\[
\||z_0|^{2/d}H^*g_{\mathbf{C}\times V_1}-g\|_{C^{k,\alpha}_{0,0}} < \epsilon.
\]
\end{prop}

From the above two propositions we have the following:

\begin{prop}
\label{tangentconeisx0}
Let $\epsilon > 0$. If $D$ is sufficiently large, then there are
$(D\epsilon)$-Gromov-Hausdorff approximations between the annular
regions
\[
X_1^D = (X_1,\omega) \cap \{D^{1/2} < \rho < D \}
\]
and
\[
X_0^D = (X_0, \omega_0) \cap \{D^{1/2} < \rho < D \}
\]
Recall that $X_0 = \mathbf{C} \times V_0$ is equipped with the product
metric $\omega_0 = \ddbar(|z|^2 + r^2)$. Consequently, the tangent
cone of $(X_1,\omega)$ at infinity is $(X_0,\omega_0)$.
\end{prop}

This is slightly different from Proposition~9 in \cite{Sz19}. Since
the above result is crucial for obtaining the asymptotic behavior of
the distance function of $\omega$, we give a detailed proof here.

\begin{proof}[Proof of Proposition~\ref{tangentconeisx0}]
Given $\epsilon > 0$, the goal is to construct a
$(D\epsilon)$-Gromov-Hausdorff approximation $G: X_1^D \to X_0^D$. Let
$\Lambda > 0$. Write $S_\Lambda = \{ R < \Lambda \rho^{1/d}
\}$. Recall that $S_\Lambda$ denotes a region that is close to the
singular ray of $X_0$. Then we can decompose $X_1^D$ into
$X_1^D \setminus S_\Lambda$ and $X_1^D \cap S_{2\Lambda}$.

First we work on $X_1^D \setminus S_\Lambda$. Recall from
Proposition~\ref{awaysingular} that once $\Lambda$ is sufficiently
large, the nearest point projection
$G: X_1^D \setminus S_\Lambda \to X_0^D$ is a diffeomorphism onto its
image, and the error in the metric is $|g - G^*g_0|_g < \epsilon$. Let
$x_1, x_2 \in X_1^D \setminus S_\Lambda$, and let $\gamma$ be a curve
in $X_1^D \setminus S_\Lambda$ connecting $x_1$ and $x_2$. Then the
error in the length is given by
\begin{equation}
\label{eq:length}
|\mathrm{length}_g(\gamma)-\mathrm{length}_{g_0}(\gamma)| \le \mathrm{length}_{g_0}(\gamma)\epsilon.
\end{equation}
It follows that
\begin{align*}
d_{X_1^D}(x_1,x_2) &\le d_{X_0^D}(G(x_1),G(x_2))(1+\epsilon) \\
&\le d_{X_0^D}(G(x_1),G(x_2)) + 2D\epsilon.
\end{align*}
The second inequality uses the fact that $X_0$ is a cone. To get the
reverse inequality, we can use \eqref{eq:length} again and get
\[
(1-\epsilon)d_{X_0^D}(G(x_1),G(x_2)) \le \mathrm{length}_g(\gamma).
\]
However, we cannot yet take the infimum of the right hand side among
all curves connecting $x_1$ and $x_2$, as the minimal geodesic
connecting $x_1$ and $x_2$ may pass through $X_1^D \cap
S_{\Lambda}$. To $d_{X_1^D}(x_1,x_2)$ is not too much smaller than the
right hand side, we turn to the study on $X_1^D \cap S_{2\Lambda}$.

On $X_1^D \cap S_{2\Lambda}$, we define $G: X_1^D \to X_0^D$ by the
projection
\[
  \mathrm{pr}_1: X_1^D \subset \mathbf{C} \times \mathbf{C}^n \to
  \mathbf{C} \subset X_0
\]
onto the singular ray of $X_0$. Proposition~\ref{nearsingular} says
that there is a map
$H: X_0^D \cap S_{2\Lambda} \to \mathbf{C} \times V_1$ with
$\mathrm{pr}_1 \circ H = \mathrm{pr}_1$ such that
$|g - H^*g_{\mathbf{C}\times V_1}|_g \le \epsilon$. Consequently, for
$x_1, x_2$ in this region, any curve $\gamma$ connecting $x_1$ and
$x_2$ satisfies
\[
\mathrm{length}_g(\gamma) \ge (1-\epsilon)\mathrm{length}_{\mathbf{C}} (G\circ \gamma)
\ge (1-\epsilon)d_{\mathbf{C}}(G(x_1),G(x_2)).
\]
To take the infimum of the left hand side, note that the shortest
curve connecting $x_1$ and $x_2$ in $X_1^D$ will remain in the region
$S_{2\Lambda}$, since on the ``annular region''
$S_{2\Lambda}\setminus S_\Lambda$ the metric can be made arbitrarily
close to the cone metric $\omega_0$ by letting $\Lambda$ and $D$ be
sufficiently large. So we have
\begin{equation}
d_{X_1^D}(x_1,x_2) \ge d_{\mathbf{C}}(G(x_1),G(x_2))-2D\epsilon.
\end{equation}

To get the reverse inequality, we write
$H(x_1) = (z_1, p_1), H(x_2) = (z_2, p_2)$ with $z_i \in \mathbf{C}$ and
$p_i \in V_1$. From the error in the metric we get
\begin{align*}
d_{X_1^D}(x_1,x_2) &\le d_{\mathbf{C} \times V_1}(H(x_1), H(x_2)) (1+ \epsilon) \\
&\le (d_{\mathbf{C}}(z_1, z_2) + d_{V_1}(p_1,p_2))(1+\epsilon) \\
&\le (d_{\mathbf{C}}(z_1, z_2) + d_{V_1}(o, p_1)) + d_{V_1}(o,p_2))(1+\epsilon).
\end{align*}
Here the second inequality follows from the Pythagorean theorem, and
$o$ is a fixed point in $V_1$. Since $d_{V_1}(o, \cdot)$ is equivalent
to $R$, we can estimate
\begin{align}
\label{eq:epsilondense}
d_{V_1}(o,p_1)) \le CR \le C\Lambda D^{1/d-1} D \ll D\epsilon
\end{align}
by choosing $D$ sufficiently large. We conclude that
\begin{align}
\label{eq:singularraydistance}
d_{X_1^D}(x_1,x_2) \le d_{\mathbf{C}}(z_1, z_2) + 2D\epsilon.
\end{align}

We now come back to the region $X_1^D \cap S_\Lambda$. Again let
$x_1, x_2 \in X_1^D \cap S_\Lambda$. Let $\gamma$ be the shortest
curve in $X_1^D$ connecting $x_1$ and $x_2$. Let $x_1'$ be the first
point of $\gamma$ entering the region $S_{\Lambda}$ and let $x_2'$ be
the last point exiting $S_{\Lambda}$. If $\gamma_1$ is the shortest
curve connecting $x_1, x_1'$, then
\begin{align*}
d_{X_1^D}(x_1,x_1') = \mathrm{length}_g(\gamma_1) \ge d_{X_0^D}(G(x_1),G(x_1')) - D\epsilon
\end{align*}
by \eqref{eq:length}. The similar inequality holds for
$d_{X_1^D}(x_2, x_2')$. We then have
\begin{align*}
d_{X_1^D}(x_1,x_2) &= d_{X_1^D}(x_1,x_1') + d_{X_1^D}(x_2,x_2') + d_{X_1^D}(x_1',x_2') \\
&\ge (d_{X_0^D}(G(x_1), G(x_1'))-D\epsilon) \\
&+ (d_{X_0^D}(G(x_2), G(x_2'))-D\epsilon) \\
&+ (d_{X_0^D}(G(x_1'), G(x_2'))-2D\epsilon) \\
&\ge d_{X_0^D}(G(x_1), G(x_2)) - 4D\epsilon
\end{align*}
using the triangle inequality and \eqref{eq:singularraydistance}.

Finally, $G(X_1^D)$ is clearly $(D\epsilon)$-dense away from the
singular ray. That $G(X_1^D)$ is $(D\epsilon)$-dense near the singular
ray follows from the estimate \eqref{eq:epsilondense}. To get the
inverse Gromov-Hausdorff approximation, away from the singular ray we
can use the nearest point projection to map into $X_1^D$. At a point
$(z, p) \in X_0$ near the singular ray, we can first map it to
$(z, o)$, where $o \in V_1$ is a fixed point, and then map $(z,o)$
into $X_1^D$ using $H^{-1}$.
\end{proof}

The following corollary will be useful in Section~\ref{sec:ds}.

\begin{cor}
\label{cor:asymptotics}
Let $d$ denote the distance function of $(X_1,\omega)$ and let
$o \in X_1$ be a fixed point. Then $d(o, \cdot)$ is uniformly
equivalent to $\rho$. Moreover, we have
\[
\lim_{\rho(x) \to \infty} \frac{d(o, x)^2}{|z|^2+r^2} = 1.
\]
\end{cor}

\begin{proof}
Write $\tilde\rho^2 = |z|^2 + r^2$. Assume for now that $o \in X_1$ is
the origin, and let $x \in X_1$, which we will let
$D = \rho(x) \to \infty$. First we note that by concatenating larger
and larger annuli of the form $(2^i, 2^{2i})$,
Proposition~\ref{tangentconeisx0} implies that the function
$d(o, \cdot)$ is equivalent to $\rho$. Since $\rho$ and $\tilde\rho$
are homogeneous of degree $2$, they are equivalent, too.

Let $x' \in X_1$ be on the minimal geodesic connecting $o$ and $x$
such that $\rho(x') = D^{1/2}$. By Proposition~\ref{tangentconeisx0},
for any $\epsilon > 0$ we have for sufficiently large $D$,
\begin{equation}
\label{eq:bigineq}
\begin{aligned}
d(o,x') + d_{X_0}(G(x'), G(x)) - D\epsilon &< d(o,x) = d(o,x') + d(x',x) \\
&< d(o,x') + d_{X_0}(G(x'), G(x)) + D\epsilon,
\end{aligned}
\end{equation}
where $G$ is the $(D\epsilon)$-Gromov-Hausdorff approximation given in
Proposition~\ref{tangentconeisx0}. Recall that away from the singular
ray, $G$ is given by the nearest point projection with respect to the
cone metric $\ddbar \rho^2$ on $\mathbf{C}^n$, and near the singular
ray we have $|z| \sim \rho$ and $G$ is given by the projection onto
the singular ray. It follows that $\rho(Gx) \sim \rho(x) = D$, and so
$\tilde\rho(x) \sim D$. As $D \to \infty$, the distance of $x$ and
$G(x)$ with respect to the scaled down cone metric
$D^{-2}\ddbar \rho^2$ converges to $0$. It follows that
\begin{equation}
\label{eq:rhoprime}
\frac{\tilde\rho(G(x))}{\tilde\rho(x)} \to 1
\end{equation}
as $D \to \infty$.

Dividing the inequality \eqref{eq:bigineq} by $\tilde\rho(x)$, we
estimate the terms as follows:
\begin{align*}
\frac{d(o,x')}{\tilde\rho(x)} \sim \frac{D^{1/2}}{D} = D^{-1/2}, \\
\frac{d_{X_0}(G(x'),G(x))}{\tilde\rho(x)} \to 1,
\end{align*}
as $D \to \infty$. Here the second estimate follows from the cosine
law of the cone metric on $X_0$ and \eqref{eq:rhoprime}. Letting
$D \to \infty$ we get the desired result. For arbitrary fixed point
$o \in X_1$ the same result follows by an application of the triangle
inequality.
\end{proof}

Finally, we recall the technical heart of \cite{Sz19}, the
invertibility of the Laplacian in weighted spaces:

\begin{prop}
Suppose that we choose $\tau \in (4-2n,0)$ (recall that $n$ is the
complex dimension of $X_1$) and $\delta$ avoids a discrete set of
indicial roots. For sufficiently large $A>0$ the Laplacian
\[
\Delta: C^{2,\alpha}_{\delta,\tau}(\rho^{-1}[A,\infty),\omega) \to
C^{0,\alpha}_{\delta-2,\tau-2}(\rho^{-1}[A,\infty),\omega)
\]
is surjective with inverse bounded independently of $A$.
\end{prop}

The idea of the proof is to cover $X_1$ (outside a big compact set) by
the open subset $\mathcal{U}$ and open subsets of types $\mathcal{V}$
near the singular rays. On each such open set, the model Laplacian is
invertible with respect to the corresponding model weighted
space. Then one construct a parametrix by patching local inverses
together using cutoff functions.

\section{Construction of new Calabi-Yau metrics}\label{sec:x1b}

We now turn to constructing a new family of Calabi-Yau metrics on
$\mathbf{C}^3$, building on the results in the previous section. Similar
to the construction of Calabi-Yau metrics on $X_1$ as in \cite{Sz19},
we will consider the following family of hypersurfaces
\[
X_{1,b} = \{ z + by + x_1^2 + x_2^2 + y^3 = 0 \} \subset \mathbf{C}^4,
\]
where $b \in \mathbf{C}$. More generally, we could consider
\[
X_{a,b} = \{ az + by + x_1^2 + x_2^2 + y^3 = 0 \} \subset \mathbf{C}^4,
\]
where $a \ne 0 \in \mathbf{C}$ and $b \in \mathbf{C}$. The effect of
$a$ can be taken care of by rescaling the metric. So we will assume
$a=1$. Later in Section~\ref{sec:ds} we will give a detailed
explanation why the following construction of Calabi-Yau metrics on
$X_{a,b}$ would possibly give all the Calabi-Yau metrics on
$\mathbf{C}^3$ with tangent cone $\mathbf{C} \times A_2$ at infinity.

Let
\[
\Omega = dx_1 \wedge dx_2 \wedge dy
\]
be the holomorphic volume form on $X_{1,b}$. The rest of the section
is dedicated to proving the following:

\begin{thm}\label{thm:construction}
  There exists a complete K\"ahler metric $\omega_{1,b}$ on $X_{1,b}$
  such that
  \[
    \omega_{1,b}^3 = \sqrt{-1} \Omega\wedge\overline\Omega,
  \]
  and that the tangent cone at the infinity given by
  $\mathbf{C} \times A_2$.
\end{thm}

Let
\[
\Phi = |z|^2 + \gamma_1(R\rho^{-\alpha})r^2 +
\gamma_2(R\rho^{-\alpha})|z|^{2/d}\phi(z^{-1/d}\cdot (x,y))
\]
be the K\"ahler potential of the approximate solution on $X_1$
constructed in the previous section. The strategy is to use the
nearest point projection
$G: X_1 \cap \{ \rho > A\} \to X_{1,b} \cap \{ \rho > A \}$ with
respect to the ambient cone metric $\ddbar \rho^2$ for large enough
$A>0$, to pull back the volume form
$\sqrt{-1}\Omega\wedge\overline\Omega$ as well as the complex
structure $J$ on $X_{1,b}$, and solve
\[\label{eq:21}
  (\bddbar (\Phi + u))^3 = \sqrt{-1}\Omega_b \wedge \overline\Omega_b.
\]
Here $\partial_b$ and $\bar\partial_b$ are the partial differentials
with respect to the complex structure $J_b = G_*J (G^{-1})_*$, and
$\Omega_b = G^*\Omega$ is the pullback of the holomorphic volume
form. Once this is done, we push forward this metric using $G$ to
$X_{1,b}$ and obtain a Calabi-Yau metric on $X_{1,b}$ outside a large
compact subset. Then we extend it to a K\"ahler metric on $X_{1,b}$
which is Ricci-flat outside a large compact subset. We can then apply
Hein's version of the Tian-Yau perturbation theorem \cite{Hein} to
perturb it again to a genuine Calabi-Yau metric on $X_{1,b}$.

\begin{remark}
  One could try to write down an explicit approximate solution on
  $X_{1,b}$ without relying on the nearest point projection, and apply
  the techniques in the previous section directly on $X_{1,b}$, but
  then an issue is that the fibration is non-trivial away from the
  singular fibers. This potentially would make the analysis harder. We
  use the nearest point projection because near the singular ray and
  far from the singular fibers, we are still comparing the geometry of
  $X_{1,b}$ to the geometry of $\mathbf{C} \times V_1$. See the proof
  of Proposition~\ref{prop:bdecay} below.
\end{remark}

The nearest point projection $G: X_1 \to X_{1,b}$ is only defined
outside compact subsets containing the origin $0 \in \mathbf{C}^4$, as
the cone metric $\ddbar \rho^2$ is singular at $0$ (and also singular
along the singular rays $\mathbf{C} \subset \mathbf{C}^4$). Recall that
scaling down the metric amounts to making the coordinate change
$z \to D^{-1}z$, $x \to D^{-1} \cdot x$. One might be tempted to
conclude that the error going from $X_1$ to $X_{1,b}$ is
$O(b\rho^{-4})$ by comparing the defining equations. If this were
true, then we may apply Hein's perturbation theorem directly to
perturb the Calabi-Yau metric $\omega$ on $X_1$ to a (pullback of)
Calabi-Yau metric on $X_{1,b}$. Unfortunately this is not the case, as
both $X_1$ and $X_{1,b}$ converges to $X_0$, whose singular set is
complex one-dimensional. To get meaningful $C^{k,\alpha}$ bounds of
the errors introduced by the nearest point projection, we need to
apply the region analysis in Proposition~\ref{prop:adecay} in the
previous section, comparing the geometry in each region to those of
different model spaces.

\subsection{Decay of the Ricci potential}

Let us write $\omega_b = \bddbar \Phi$ as the approximate solution. As
mentioned above, we want to solve \eqref{eq:21} on
$X_1 \cap \{ \rho > A\}$ for large enough $A$. To solve for $u$, we
want to ensure that the Ricci potential
\[
h = \log \frac{\omega_b^3}{\sqrt{-1}\Omega_b \wedge \overline\Omega_b}
\]
has fast enough decay in order to apply the technical results
discussed in the previous section. We have the following
generalization of Proposition~\ref{prop:adecay} in our
$\mathbf{C} \times A_2$ case.

\begin{prop}
\label{prop:bdecay}
Fix $\alpha \in (1/d,1)$. The form $\omega_b$ defines a K\"ahler
metric with respect to the deformed complex structure $J_b$ on the
$X_1 \cap \{\rho > P\}$, for sufficiently large $P$ (depending on
$b$). For suitable constants $\kappa, C_i >0$ and weight
$\delta < 2/d$, the Ricci potential $h$ of $\omega_b$ with respect to
$G^*(\sqrt{-1}\Omega_b\wedge\overline\Omega_b)$ and the error in the
complex structure satisfy, for large $\rho$,
\begin{align*}
|\nabla^i h|_\omega,
|\nabla^i(\omega_b-\omega)|,|\nabla^i(J_b-J)|_\omega < \max\{1,b\}
\begin{cases}
C_i \rho^{\delta-2-i} & \text{if } R > \kappa\rho \\
C_i \rho^{\delta}R^{-2-i} & \text{if } R \in (\kappa^{-1}\rho^{1/d}, \kappa \rho) \\
C_i \rho^{\delta-2/d-i/d} & \text{if } R < \kappa^{-1}\rho^{1/d}.
\end{cases}
\end{align*}
In fact, since $d=6$, we can choose $\delta \in [-1/3,1/3)$. In terms
of the weighted spaces defined in the previous section, we have that
\[
\|h\|_{C^{k,\alpha}_{\delta-2,-2}},
\|\omega_b-\omega\|_{C^{k,\alpha}_{\delta-2,-2}},
\|J_b-J\|_{C^{k,\alpha}_{\delta-2,-2}} \le C_k\max\{1,b\}
\]
for a uniform constant $C_k>0$.
\end{prop}

\begin{proof}
The proof is very similar to Proposition~\ref{prop:adecay} before. The
main difference is that in this case the complex structure as well as
the holomorphic volume form are deformed. As a result the Ricci
potential is given by
\begin{align*}
h &= \log \frac{\omega_b^3}{\sqrt{-1}\Omega_b \wedge \overline\Omega_b} \\
&=
\log \frac{\omega^3}{\sqrt{-1}\Omega \wedge \overline\Omega} +
\log \frac{\omega_b^3}{\omega^3} + 
\log \frac{\sqrt{-1}\Omega\wedge\overline\Omega}{\sqrt{-1}\Omega_b \wedge \overline\Omega_b}.
\end{align*}
Here we recall that $\Omega$ is the holomorphic volume form on
$X_1$. Thus we will have additional errors introduced by the change in
the complex structure as well as the change in the volume form. For
the metric, we can estimate the error by
\[
\omega_b - \omega = d (J_b-J) d\Phi.
\]
Since $\Phi$ has growth rate $2$, it follows that the error in the
metric is dominated by the error in the change of the complex
structure. We perform the region analysis as in the proof of
Proposition~\ref{prop:adecay}.

\textbf{Region I}: Suppose $R > \kappa\rho$ and $\rho \in (D/2, 2D)$
for some large $D$. Since $R > (\kappa/2)D$, we are uniformly away
from the singular rays. We study the scaled metric $D^{-2}\omega$ in
terms of the rescaled coordinates $\tilde{z} = D^{-1}z$,
$\tilde{x} = D^{-1} \cdot x$. The equation of $X_1$ becomes
\[
D^{1-d}\tilde{z} + f(\tilde{x}) = 0,
\]
and the equation of $X_{1,b}$ becomes
\[
D^{1-d}\tilde{z} + bD^{2-d}\tilde{y} + f(\tilde{x}) = 0.
\]
Thus the extra error is of order $bD^{2-d}$. We can choose any
$\delta$ such that $\delta-2 > 2-d$. Since $d = 6$, we can
make $\delta < 0$.

\textbf{Region II}: Suppose now that $R \in (K/2, 2K)$ for some
$K < \kappa \rho$, $K/2 > 2\rho^\alpha$ and $\rho \in (D/2, 2D)$. In
this case $\rho$ is comparable to $|z|$. We assume that for some fixed
$z_0$ we have $|z-z_0| < K$. We now scale the metric by $K$, and define
\[
\tilde{z} = K^{-1}(z-z_0),\:\:\: \tilde{x} = K^{-1} \cdot x, \:\:\: \tilde{r} =
K^{-1}r.
\]
The equation of $X_1$ is
\[
K^{-d}(K\tilde{z} + z_0) + f(\tilde{x}) = 0,
\]
while the equation of $X_{1,b}$ is
\[
K^{-d}(K\tilde{z} + z_0) + bK^{2-d}\tilde{y} + f(\tilde{x}) = 0.
\]
Since $|\tilde{y}| \sim 1$, thus the extra error in the Ricci potential
is of order $bK^{2-d}$. Since $d=6$ and $K > 4\rho^\alpha$, we have
\[
bK^{4-d}K^{-2} < bCD^{(4-d)\alpha}K^{-2}
\]
for a constant $C$. We can choose $\delta < 0$ such that
$(4-d)\alpha < \delta$. If $\alpha$ is close to $1$ then we can choose
$\delta = -1$.

\textbf{Region III}: Suppose $R \in (K/2, 2K)$,
$K \in (\rho^\alpha, 2\rho^\alpha)$ and $\rho \in (D/2, 2D)$. Thus
$|z|$ is comparable to $D$. We are in the gluing region. We scale as
in Region II. The equation of $X_1$ becomes
\[
K^{-d}(K\tilde{z} + z_0) + f(\tilde{x}) = 0,
\]
and the equation of $X_{1,b}$ becomes
\[
K^{-d}(K\tilde{z} + z_0) + bK^{2-d}\tilde{y} + f(\tilde{x}) = 0.
\]
The extra error in the Ricci potential is then again of order
$bK^{2-d}$. Since $K \sim D^\alpha$, we can estimate it as follows:
\[
bK^{4-d}K^{-2} < bCD^{(4-d)\alpha}K^{-2}.
\]
So here we can choose $0 > \delta > (4-d)\alpha$.

\textbf{Region IV}: Suppose now that $R \in (K/2, 2K)$,
$K \in (\kappa\rho^{1/d}, \rho^\alpha/2)$, and $\rho \in (D/2,
2D)$. Then we have $|z| \sim D$. We scale in the same way as in
Regions II, III. The equation of $X_1$ is
\[
K^{-d}(K\tilde{z} + z_0) + f(\tilde{x}) = 0,
\]
and we are comparing $X_1$ to $\mathbf{C} \times V_{K^{-d}z_0}$, given by the
equation
\[
K^{-d}z_0 + f(\tilde{x}) = 0.
\]
On the other hand the error going from $X_{1,b}$ to $X_1$ is still of
order $bK^{2-d}$. Since $K > \kappa \rho^{1/d}$, we get
\[
K^{4-d}K^{-2} \le bCD^{4/d-1} K^{-2}.
\]
Since $4/d-1 < 0$, we can choose $\delta < 0$.

\textbf{Region V}: Suppose that $R < 2\kappa^{-1}\rho^{1/d}$,
$\rho \in (D/2, 2D)$. Then $|z|$ is comparable to $D$. Fix $z_0$ and
let $z$ be very close $z_0$. We scale by $|z_0|^{1/d}$:
\[
\tilde{z} = z_0^{-1/d}(z-z_0), \:\:\: \tilde{x} = z_0^{-1/d}\cdot x,
\:\:\: \tilde{r} = |z|^{-1/d}r.
\]
So we have $|\tilde{z}|, |\tilde{r}| < C$. We are near the singular
rays. So we compare $X_1$:
\[
z_0^{1/d-1}\tilde{z} + 1 + f(\tilde{x}) = 0
\]
with $\mathbf{C} \times V_1$:
\[
1 + f(\tilde{x}) = 0.
\]
On the other hand, the equation of $X_{1,b}$ becomes
\[
z_0^{1/d-1}\tilde{z} + 1 + b z_0^{2/d-1}\tilde{y} + f(\tilde{x}) = 0.
\]
So the extra error in this case is
$b|z_0|^{(2-d)(1/d)} \le bCD^{2/d-1} \le bCD^{\delta - 2/d}$, where we
choose $0 > \delta \ge 4/d -1$.
\end{proof}

As indicated in the proof, the decay rate of the error introduced by
the nearest point projection is slower than quadratic in the region
close to the singular rays, so we cannot apply Hein's perturbation
theorem directly. But as the proposition concludes, we still have good
decay rates that allow us to improve upon using the contraction
mapping principle as in \cite{Sz19}. We first need to take care of the
fact that in our case, the Laplacian is also perturbed:

\begin{lemma}
\label{lemma:laplacianinverse}
Let $\tau \in (-2,0)$, and let $\delta$ avoids a discrete set of
indicial roots. The Laplacian $\Delta_b$ with respect to the metric
defined by $\omega_b$ is a map from
$C^{2,\alpha}_{\delta,\tau}(\rho^{-1}[A,\infty))$ to
$C^{0,\alpha}_{\delta-2,\tau-2}(\rho^{-1}[A,\infty))$ with bounded
right inverse when $A$ is sufficiently large.
\end{lemma}

\begin{proof}
Let
$P: C^{0,\alpha}_{\delta-2,\tau-2}(\rho^{-1}[A,\infty)) \to
C^{2,\alpha}_{\delta,\tau}(\rho^{-1}[A,\infty))$ be the right inverse
for $\Delta$. For $u \in C^{k,\alpha}_{\delta,\tau}$, by direct
computation we have
\begin{align*}
\|\Delta_bu-\Delta u\|_{C^{0,\alpha}_{\delta-2,\tau-2}}
\le \|\nabla(g_b-g) \ast \nabla u\|_{C^{0,\alpha}_{\delta-2,\tau-2}}
+ \|(g_b-g)\ast \nabla^2 u\|_{C^{0,\alpha}_{\delta-2,\tau-2}}.
\end{align*}
Using the properties of the weighted norms in
Propsition~\ref{prop:weightedprop}, we have
\begin{align*}
\|\nabla(g_b-g)\ast \nabla u\|_{C^{0,\alpha}_{\delta-2,\tau-2}}
&\le C\|\nabla(g_b-g)\|_{C^{0,\alpha}_{-1,-3}}\|\nabla u\|_{C^{0,\alpha}_{\delta-1,\tau+1}} \\
&\le C\|g_b-g\|_{C^{1,\alpha}_{0,-2}}\|u\|_{C^{1,\alpha}_{\delta,\tau+2}} \\
&\le C\|g_b-g\|_{C^{2,\alpha}_{0,-2}}\|u\|_{C^{2,\alpha}_{\delta,\tau+2}} \\
&\le C\|g_b-g\|_{C^{2,\alpha}_{0,-2}}\|u\|_{C^{2,\alpha}_{\delta,\tau}}.
\end{align*}
Similarly,
\begin{align*}
\|(g_b-g)\ast \nabla^2 u\|_{C^{0,\alpha}_{\delta-2,\tau-2}}
&\le C\|g_b-g\|_{C^{0,\alpha}_{0,-2}}\|\nabla^2 u\|_{C^{0,\alpha}_{\delta-2,\tau}} \\
&\le C\|g_b-g\|_{C^{1,\alpha}_{2,-2}}\|u\|_{C^{2,\alpha}_{\delta,\tau+2}} \\
&\le C\|g_b-g\|_{C^{2,\alpha}_{0,-2}}\|u\|_{C^{2,\alpha}_{\delta,\tau}}.
\end{align*}
It follows that
\begin{align*}
\|\Delta_bu-\Delta u\|_{C^{0,\alpha}_{\delta-2,\tau-2}} \le
C\|g_b-g\|_{C^{2,\alpha}_{0,-2}}\|u\|_{C^{2,\alpha}_{\delta,\tau}}
\end{align*}
for a uniform constant $C > 0$. By Proposition~\ref{prop:bdecay},
$\|g_b-g\|_{C^{2,\alpha}_{0,-2}(\rho^{-1}[A,\infty)))}$ can be made
arbitrarily small once $A \gg 1$. It follows that
\begin{align*}
\|u - \Delta_b P u\|_{C^{0,\alpha}_{\delta-2, \tau-2}}
&\le \|\Delta_bP u-\Delta P u\|_{C^{0,\alpha}_{\delta-2,\tau-2}} \\
&\le
C\|g_b-g\|_{C^{2,\alpha}_{0,-2}}\|u\|_{C^{2,\alpha}_{\delta,\tau}}
\ll \|u\|_{C^{2,\alpha}_{\delta,\tau}}.
\end{align*}
It follows that $\Delta_b$ admits a bounded right inverse.
\end{proof}

\subsection{Perturbing to genuine solution}

We use the approximate solution $\omega$ on $X_1$ and the weighted
spaces defined in the previous section. Recall that our goal is to
first solve \eqref{eq:21} on $X_1\cap \{\rho > A\}$ for large
$A$. Define
\[
\mathcal{B} = \{ u \in C^{2,\alpha}_{\delta, \tau}\mid
\|u\|_{C^{2,\alpha}_{\delta, \tau}} < \epsilon_0\},
\]
where $\tau$ is now chosen to be close to $0$ and $\epsilon_0$ is
sufficiently small such that $\omega + \ddbar u$ has the same tangent
cone at infinity as $\omega$. Consider the following operator
\begin{align*}
F: \mathcal{B} &\to C^{0,\alpha}_{\delta-2,\tau-2}(\rho^{-1}[A,\infty)) \\
u &\mapsto \log \left.\frac{(\tilde\omega+ \bddbar
  u)^3}{{\sqrt{-1}\Omega_b\wedge\overline\Omega_b}}\right|_{\rho^{-1}[A,\infty)},
\end{align*}
and write
\begin{equation}
\label{eq:19}
F(u) = F(0) + \Delta_b u + Q(u),
\end{equation}
where $Q$ is the nonlinear part of $F$. Here $F(0)=h$ is given by the
Ricci potential defined above. The goal is to find $u \in \mathcal{B}$ such
that $F(u) = 0$, or equivalently
\begin{equation}
\label{eq:24}
\Delta_b u = -F(0) -Q(u).
\end{equation}
Let $P$ be the right inverse for $\Delta_b$ in
Lemma~\ref{lemma:laplacianinverse}. Define the map
\[
N(u) = P(-F(u)-Q(u)).
\]
Then finding a solution to \eqref{eq:24} is the same as finding a
fixed point of $N$. Note that we have a uniform bound for $P$
independent of sufficiently large $A$. Thus we can enlarge $A$ when
needed. From an explicit formula for $Q$ (e.g. expand
$\log \det (I+A)-\mathrm{tr} A$ using eigenvalues for $A$), we see
that if
\[
\|\bddb u\|_{C^{0,\alpha}_{0,0}},\|\bddb v\|_{C^{0,\alpha}_{0,0}} \ll 1,
\]
then we have the estimate
\[
\|Q(u)-Q(v)\|_{C^{0,\alpha}_{\delta-2,\tau-2}} \le
C(\|\bddb u\|_{C^{0,\alpha}_{0,0}}+\|\bddb v\|_{C^{0,\alpha}_{0,0}})
\|\bddb(u-v)\|_{C^{0,\alpha}_{\delta-2,\tau-2}}.
\]
To estimate $\|\bddb u\|_{C^{0,\alpha}_{0,0}}$ in terms of the norm of
$u$, we have
\begin{equation}
\begin{aligned}
\|\bddb u\|_{C^{0,\alpha}_{0,0}} &\le \|\ddb u\|_{C^{0,\alpha}_{0,0}} +
\|(\bddb-\ddb) u\|_{C^{0,\alpha}_{0,0}} \\
&\le C(1 + \|J_b-J\|_{C^{2,\alpha}_{-2,-2}})\|u\|_{C^{2,\alpha}_{2,2}} \\
&\le C\max\{1,b\}\|u\|_{C^{2,\alpha}_{2,2}}
\end{aligned}
\end{equation}
by Proposition~\ref{prop:bdecay}. Since we have
\[
\rho^{\delta}w^\tau \le C\rho^{\delta-2+(\tau-2)(1/d-1)}\rho^2 w^2,
\]
which implies
\[
\|u\|_{C^{2,\alpha}_{2,2}} \le C\|u\|_{C^{2,\alpha}_{\delta,\tau}},
\]
by choosing $\epsilon_0 < C\max\{1,b\}^{-1}$ for a uniform constant
$C > 0$ we have
\[
\|N(u)-N(v)\|_{C^{2,\alpha}_{\delta,\tau}} < \frac{1}{2}\|u-v\|_{C^{2,\alpha}_{\delta,\tau}}
\]
for $u,v \in \mathcal{B}$; i.e. $N$ is a contraction mapping. It
remains to ensure that $N$ maps $\mathcal{B}$ into
$\mathcal{B}$. First we note that by the estimates of
Proposition~\ref{prop:bdecay} we have
$F(0) \in C^{0,\alpha}_{\delta'-2,\tau-2}$ for some $\delta' < \delta$
(increase $\delta$ if necessary) sufficiently close to $\delta$. It
follows that
\[
\|F(0)\|_{C^{0,\alpha}_{\delta,\tau-2}(\rho^{-1}[A,\infty))} <
CA^{\delta'-\delta}.
\]
Combining the estimates above, we have that if $u \in \mathcal{B}$, then
\begin{align*}
\|N(u)\|_{C^{2,\alpha}_{\delta,\tau}}
&\le \|N(0)\|_{C^{2,\alpha}_{\delta,\tau}} + \|N(u)-N(v)\|_{C^{2,\alpha}_{\delta,\tau}} \\
&\le \|F(0)\|_{C^{0,\alpha}_{\delta,\tau-2}(\rho^{-1}[A,\infty))} +
\frac{1}{2}\|u\|_{C^{2,\alpha}_{\delta,\tau}} \\
&\le \max\{1,b\}CA^{\delta'-\delta} + \frac{\epsilon_0}{2}.
\end{align*}
We see that to make $N$ maps into $\mathcal{B}$, we need to remove larger and
larger compact subsets as $b$ gets larger. In sum we can make $N$ a
contraction mapping by choosing $A$ sufficiently large (depending on
$b$). Thus there exists
$u \in C^{k,\alpha}_{\delta,\tau}(\rho^{-1}[A,\infty))$ with
$\|u\|_{C^{k,\alpha}_{\delta,\tau}} < \epsilon_0$ such that
\[
(\omega_b + \bddbar u)^3 = \sqrt{-1} \Omega_b \wedge \overline\Omega_b.
\]
on $\rho^{-1}[A,\infty)$. Pushing forward to $X_{1,b}$, we have
\[
(\ddbar ((\Phi + u)\circ G^{-1}))^3 =
\sqrt{-1}\Omega \wedge\overline\Omega
\]
on $\rho^{-1}[A,\infty)$. We can modify the K\"ahler potential so that
it defines a K\"ahler potential $\tilde{\Phi}_b$ on $X_{1,b}$ such
that $\ddbar \tilde{\Phi}_b$ agrees with
$\ddbar ((\Phi + u)\circ G^{-1})$ on $\rho^{-1}[2A,\infty)$. Set
$\tilde{\omega}_b = \ddbar \tilde{\Phi}_b$.

We now apply Hein's version~\cite{Hein} of the Tian-Yau
perturbation. Recall that $(X_1 \cap \rho^{-1}[A,\infty), \omega)$ is
covered by sets of type $\mathcal{U}$ and type $\mathcal{V}$ as in
Proposition~\ref{awaysingular} and Proposition~\ref{nearsingular},
respectively. Pulling back using the nearest point projection, it
follows that the same holds for
$(X_{1,b} \cap \rho^{-1}[2A,\infty),\tilde{\omega}_b)$. We can rescale
accordingly to compare in each region to the model geometries $X_0$
and $\mathbf{C} \times V_1$. The rescaled coordinates then give the
desired $C^{3,\alpha}$ coordinates. For the compact part we simply
cover it with a finite number of coordinate balls. This shows that
$(X_{1,b},\tilde{\omega}_b)$ admits a $C^{3,\alpha}$ quasi-atlas. That
$\tilde{\omega}_b$ is $\mathrm{SOB}(6)$ follows from that
$\tilde{\omega}_b$ is Ricci-flat outside a compact subset and the
tangent cone at infinity is $X_0$, which together imply maximal volume
growth by Colding's volume convergence~\cite{Colding}. We can then
apply \cite[Proposition~4.1]{Hein} to perturb $\tilde{\omega}$ to a
genuine Calabi-Yau metric $\omega_{1,b} = \ddbar \varphi_{1,b}$ on
$X_{1,b}$. When $b=0$, this recovers the Calabi-Yau metric constructed
on $X_1 = X_{1,0}$ in \cite{Sz19}.

A few notes about this construction are in order. First, this
construction should generalize to construct families of Calabi-Yau
metrics asymptotic to $\mathbf{C} \times A_k, k \ge 3$, including the
ones constructed in \cite{Sz19}. One could consider hypersurfaces in
$\mathbf{C}^4$ given by
$az + b_1y + b_2y^2 + \ldots + b_{k-2}y^{k-1} + x_1^2 + x_2^2 +
y^{k+1} = 0$, with $a \ne 0 \in \mathbf{C}$ and $b_i \in
\mathbf{C}$. Second, what we know about these metrics $\omega_{1,b}$
for now is that they are unique up to subquadratic perturbation of the
K\"ahler potential by \cite[Theorem 1.3]{CSz}. So a small perturbation
of the initial data or the choice of the right inverse of the
Laplacian does not affect the resulting metric. It is not clear at
this moment whether $\omega_{1,b}$ and $\omega_{1,b'}$ are related by
an automorphism of $\mathbf{C}^3$ up to scaling, because the
construction involves nearest point projections which are not even
holomorphic to begin with. To distinguish them we need to exploit the
explicit nature of the asymptotics.

More generally, we would like to know if the gluing construction above
gives all the Calabi-Yau metrics on $\mathbf{C}^3$ with tangent cone
$\mathbf{C} \times A_2$ at infinity. We will discuss some preliminary
results in the next section.

\section{Distinguishing the metrics}\label{sec:ds}

In this section, we conclude the proof of Theorem~\ref{thm:A2}, and
discuss some preliminary results about Conjecture~\ref{conj:A2}
below. 

Uniqueness results in singular perturbation problems are usually hard
to obtain, and very few results in the Calabi-Yau setting are
known. We would like to follow a similar strategy in \cite{Sz20} to
study the classification problem in our case. For this we first
compute subquadratic harmonic functions on the cone
$\mathbf{C} \times A_2$.

\subsection{Subquadratic harmonic functions on cones}

We first recall the following characterization of subquadratic
harmonic functions of Calabi-Yau cones $C(Y)$:

\begin{lemma}\label{lemma:HS}
  Suppose $C(Y)$ is a metric tangent cone of a non-collapsed
  Gromov-Hausdorff limit of K\"ahler-Einstein manifolds. Let $r$
  denote the radial coordinate so that $r\partial_r$ is the homothetic
  vector field. Let $J$ denote the complex structure. Suppose $u$ is a
  harmonic function on $C(Y)$. Then we have the following:
  \begin{enumerate}
  \item If $u$ is $s$-homogeneous ($r\partial_r u = s u$)
    with $s< 2$, then $u$ is pluriharmonic.
  \item If $u$ is $2$-homogeneous harmonic, then $u = u_1 + u_2$,
    where $u_1$ is pluriharmonic, and $u_2$ is
    $J(r\partial_r)$-invariant.
  \item The space of real holomorphic vector fields that commute with
    $r\partial_r$ can be written as
    $\mathfrak{p}\oplus J\mathfrak{p}$, where $\mathfrak{p}$ is
    spanned by $r\partial_r$ and vector fields of the form $\nabla u$,
    where $u$ is a $J(r\partial_r)$-invariant harmonic function
    homogeneous of degree $2$. $J\mathfrak{p}$ consists of real
    holomorphic Killing vector fields.
  \end{enumerate}
\end{lemma}

For a proof, see \cite[Lemma~3.1]{CSz} and the references therein. We
apply this lemma to systematically calculate subquadratic harmonic
functions on $C(Y)$. First we note that since $C(Y)$ is an affine
variety~\cite{LSz}, Lemma~\ref{lemma:HS} (1) and (2) imply that many
of these subquadratic harmonic functions are given by the real part of
subquadratic holomorphic functions. We are more interested in
quadratic harmonic functions whose gradients generate automorphisms of
$C(Y)$. For this we use Lemma~\ref{lemma:HS} (3) and turn to real
holomorphic vector fields. We note that $\mathfrak{p}$ has another
characterization:
\begin{align*}
\mathfrak{p} = \{ V: V \text{ is real holomorphic with linear growth
  and } JV(r^2) = 0\}.
\end{align*}

Since $C(Y)$ is an affine variety, it is useful to find $W = V-iJV$ first
and then take the real part of $W$. We follow this approach and
calculate a few examples relevant to this paper.

\begin{exmp}
  Let us consider $C(Y) = \mathbf{C} \times A_1$, defined as the
  hypersurface
  $\{x_1^2+\ldots + x_n^2 = 0\} \subset \mathbf{C} \times
  \mathbf{C}^n$, $n \ge 3$. $C(Y)$ is equipped with a Ricci-flat
  K\"ahler metric
\begin{align*}
\omega_0 = \ddb(|z|^2 + |x|^{2\frac{n-2}{n-1}}),
\end{align*}
where the coordinate $z$ has weight $1$ and the coordinates $x_i$ have
weight $(n-1)/(n-2)$. Any (complex) holomorphic vector field $W$ in
$(\mathfrak{p}\oplus J\mathfrak{p})\otimes \mathbf{C}$ is given by
\begin{align*}
W = bz\partial_z + a_{ij}x_i\partial_{x_j},
\end{align*}
where the coefficients $b$ and $a_{ij}$ are such that
$W(x_1^2+\ldots+x_n^2) = 0$ and $\mathrm{Im} W (r^2) = 0$. From these
two equations, we get that $b$ and $\lambda$ are real, and that
$a_{ij} = \sqrt{-1}b_{ij}+ \lambda\delta_{ij}$, where
$(b_{ij}) \in \mathfrak{o}(n, \mathbf{R})$. Write
$W_1 = bz\partial_z + \lambda x_i\partial_{x_i}$, and
$W_2 = \sqrt{-1}b_{ij}x_i\partial_{x_j}$. Note that
$\mathrm{Re} W_2(r^2)$ does not contain the $|z|^2$ term, so in
particular it is not proportional to $r^2$. It follows that
$\mathrm{Re} W_2(r^2)$ is a harmonic function. It remains to look at
$W_1$. For $\mathrm{Re} W_1(r^2)$ to be a harmonic function, we need
\begin{align*}
\Delta \mathrm{Re} W_1(r^2) = \Delta \left(b|z|^2 +
\lambda\frac{n-2}{n-1}|x|^{2\frac{n-2}{n-1}}\right) = 2b +
2\lambda(n-2) = 0,
\end{align*}
and so $W_1 = (n-2)z\partial_z- x_i\partial_{x_i}$. The
corresponding harmonic functions are
\begin{align*}
u_1 &= W_1(r^2) = (n-2)|z|^2-\frac{n-2}{n-1}|x|^{2\frac{n-2}{n-1}}, \\
u_2 &= W_2(r^2) = \sqrt{-1}\frac{n-2}{n-1}|x|^{\frac{-2}{n-1}}b_{ij}x_i\bar{x}_j.
\end{align*}

In \cite{Sz20}, the same result is obtained using Fourier transform in
the $\mathbf{C}$-direction.
\end{exmp}

\begin{exmp}\label{exmp:a2}
  Let $A_2$ denote the singularity
  \begin{align*}
    \{ x_1^2+x_2^2+y^3 = 0 \} \subset \mathbf{C}^3.
  \end{align*}
  Then $A_2$ is isomorphic to $\mathbf{C}^2/\mathbf{Z}_3$ via the map
  \begin{align*}
    \mathbf{C}^2 &\to \mathbf{C}^3 \\
    (z_1,z_2) &\mapsto (\frac{z_1^3+z_2^3}{2}, \frac{z_1^3-z_2^3}{2\sqrt{-1}}, \zeta z_1z_2),
  \end{align*}
  where $\zeta$ is a cubic root of $-1$. The holomorphic volume form is
  given by
  \begin{align*}
    \Omega = \frac{dx_1\wedge dx_2}{3y^2}.
  \end{align*}
  Pulling $\Omega$ back to $\mathbf{C}^2$ gives a constant multiple of
  $dz_1\wedge dz_2$. The standard flat metric on $\mathbf{C}^2$ thus
  gives the correct Calabi-Yau cone metric on $A_2$. The potential
  $r^2$ on $A_2$, using the ambient coordinates $x_1, x_2$ and $y$, is
  given by
  \begin{align*}
    r^2 =
    &\left(|x_1|^2+|x_2|^2 + \sqrt{(|x_1|^2+|x_2|^2)^2- |y|^6}\right)^{1/3} \\
    &+ \left(|x_1|^2+|x_2|^2 - \sqrt{(|x_1|^2+|x_2|^2)^2- |y|^6}\right)^{1/3}.
  \end{align*}
  This can be seen by solving a cubic equation. The complexified radial
  vector field on $\mathbf{C}^2$, $z_i\partial_{z_i}$, pushes forward to
  \begin{align*}
    3x_1\partial_{x_1}+3x_2\partial_{x_2}+2y\partial_y.
  \end{align*}
  So $x_1, x_2$ have weight $3$ and $y$ has weight $2$. Alternatively,
  the weights can be read off from the complex Monge-Amp\`ere equation.

  Any (complex) holomorphic vector field of linear growth is given by
  \begin{align*}
    W = a_{ij}x_i\partial_{x_j} + by\partial_y,
  \end{align*}
  where the coefficients $a_{ij}$ and $b$ are chosen so that
  $W(x_1^2+x_2^2+y^3) = 0$ and $\mathrm{Im} W(r^2) = 0$. It follows that
  \begin{align*}
    W = by\partial_y + \sqrt{-1}b_{ij} x_i \partial_{x_j} + c x_i\partial_{x_i},
  \end{align*}
  where $b, c$ are real with $3b+2c = 0$ and $b_{ij}$ is real and
  skew-symmetric. Thus $W_1 = \sqrt{-1}b_{ij} x_i \partial_{x_j}$, and
  $W_2 = \frac{1}{3} y\partial_y + \frac{1}{2}
  x_i\partial_{x_i}$. $W_2$ is the (complexified) radial vector. The
  space of homogeneous ($Jr\partial_r$)-invariant quadratic growth
  harmonic functions on $A_2$ is generated by
  \[
    u_1 = W_1(r^2) =
    \frac{1}{3}\sqrt{-1}b_{ij}\frac{r^2}{\sqrt{(|x_1|^2+|x_2|^2)^2-|y|^6}}x_i\bar{x}_j.
  \]  
\end{exmp}

\begin{exmp}\label{exmp:cxa2}
  We now assume that our cone is $\mathbf{C} \times A_2$. Following
  the calculations in the previous examples, it is easily seen that
  the space of ($Jr\partial_r$)-invariant homogeneous harmonic
  functions with quadratic growth on
  \[
    \mathbf{C} \times A_2 = \{x_1^2+x_2^2+y^3=0\} \subset \mathbf{C}^4 = \mathbf{C} \times \mathbf{C}^3
  \]
  is generated by
  \begin{align*}
    u_1 &= \frac{1}{3}\sqrt{-1}b_{ij}\frac{r^2}{\sqrt{(|x_1|^2+|x_2|^2)^2-|y|^6}}x_i\bar{x}_j,\\
    u_2 &= 2|z|^2 - r^2,
  \end{align*}
  where $u_2$ corresponds to the vector
  \begin{align*}
    W_2 = z\partial_z - \frac{1}{2}(2y\partial_y + 3x_i\partial_{x_i}).
  \end{align*}

  Let us consider
  \begin{align*}
    V = \mathrm{Re}(z\partial_z + \frac{1}{3} y\partial_y + \frac{1}{2} x_i \partial_{x_i}).
  \end{align*}
  Then $L_V \Omega = n\beta \Omega$, and
  \begin{align*}
    V(|z|^2 + r^2) - \beta (|z|^2 + r^2) = \frac{5}{18} u_2,
  \end{align*}
  where $\beta = 4/9$.

  $V$ generates biholomorphisms
  \[\label{eq:aut}
    \Phi_t(z, x_1,x_2, y) = (e^{t/2}z, e^{t/4}x_1, e^{t/4}x_2,
    e^{t/6}y),
  \]
  where $t \in \mathbf{C}$.

  Let us recall the notion $X_0 = \mathbf{C} \times A_2$ and $X_{1,b}$
  in the previous sections. The automorphisms $\Phi_t$ fix $X_0$, and
  move $X_{1,b}$:
  \[
    \Phi_t(X_{1,b}) = X_{1,e^{t/3}b}.
  \]
  Thus the only $X_{1,b}$ that is fixed by $\Phi_t$ is $X_{1,0}$.

  The effect of the automorphism $\Phi_t$ on the cone metric and the
  holomorphic volume form on $X_0$ is seen as
  \begin{align*}
    \Phi_t^* (|z|^2 + r^2) &= e^t |z|^2 + e^{t/6} r^2, \\
    \Phi_t^*( dx_1 \wedge dx_2 \wedge dy) &= e^{2t/3} dx_1 \wedge dx_2 \wedge dy.
  \end{align*}
  So
  \begin{align*}
    e^{-4t/9} \phi_t^* (|z|^2 + r^2) = e^{5t/9} |z|^2 + e^{-5t/18} r^2
  \end{align*}
  defines a Calabi-Yau cone metric on $X_0$ with the same volume form as
  that of $|z|^2 + r^2$. Taking Taylor expansion, we have
  \begin{align*}
    e^{-4t/9} \Phi_t^* (|z|^2 + r^2) = (|z|^2 + r^2) + \frac{5}{18}u_2 t + O(t^2).
  \end{align*}

  It follows that up to first order, perturbing the cone metric by
  $u_2$ corresponds to applying the automorphism $\Phi_t$ and
  rescaling. As in \cite{Sz20}, the reason we want to consider $V$ in
  place of $\operatorname{Re} W_2$ is that $W_2$ does not fix any of
  $X_{1,b}$. Since the automorphisms $\Phi_t$ fix the hypersurface
  $X_1=X_{1,0}$, we can still prove a result similar to \cite{Sz20}
  (see Proposition~\ref{prop:x1unique} below).
\end{exmp}

\subsection{Donaldson-Sun theory}

We now apply Donaldson-Sun theory \cite{DS17} to construct sequences
of special embeddings of $\mathbf{C}^3$ into $\mathbf{C}^4$ using
holomorphic functions with polynomial growth. The following is similar
to \cite[Proposition~3.1]{Sz20}:

\begin{prop}
\label{specialembeddings}
Suppose $X = \mathbf{C}^3$ is equipped with a Calabi-Yau metric
$\omega$ with $\mathbf{C} \times A_2$ as tangent cone at
infinity. Then there exists a sequence of holomorphic embeddings
$F_i: X \to \mathbf{C}^4$ with the following properties:
\begin{enumerate}
\item On the ball $B_i$, the map $F_i$ gives a
$\Psi(i^{-1})$-Gromov-Hausdorff approximation to the embedding
$B(0,1) \to \mathbf{C}^4$, where $B(0,1)$ is the unit ball in
$\mathbf{C} \times A_2$.

\item The image $F_i(X)$ is given by the equation
\begin{align*}
a_i z + b_i y + x_1^2 + x_2^2 + y^3 = 0,
\end{align*}
for some $a_i > 0, b_i \ge 0$. Either all $b_i = 0$ or all
$b_i \ne 0$.
\item There exists a point $o \in X$ such that $F_i(o) = 0$ for all
$i$.
\item The volume form $\omega^3$ satisfies
\begin{align*}
2^{-6i} \omega^3 = F_i^*(\sqrt{-1} \Omega \wedge \overline{\Omega}),
\end{align*}
where $\Omega = a_i^{-1}dx_1\wedge dx_2 \wedge dy$ is the holomorphic
volume form on $X_{a_i,b_i} = F_i(X)$.
\item $a_i/a_{i+1} \to 2^5$ and
$b_i/b_{i+1} = 2^{3/2}(a_i/a_{i+1})^{1/2} \to 2^4$ (when $b_i \ne 0$)
as $i \to \infty$. Furthermore, the number $b=b_ia_i^{-1/2}2^{3i/2}$
is independent of $i$ and independent of the sequence.
\end{enumerate}
We call any sequence of embeddings satisfying the above properties a
sequence of special embeddings.
\end{prop}

\begin{proof}
  The proof follows a similar strategy of \cite[Proposition
  3.1]{Sz20}.  Let $x_1, x_2, y, z$ be holomorphic functions on the
  cone $X_0 = \mathbf{C} \times A_2$ with weight $3,3,2,1$,
  correspondingly. Recall that the defining equation is given by
  $x_1^2+x_2^2+y^3=0$. Let $F_i = (x_1^i, x_2^i, y^i, z^i)$ be the
  sequence of holomorphic embeddings of $X$ into $\mathbf{C}^4$, where
  the components have weights $3,3,2,1$ respectively, such that over
  the balls $B(p_i, 1) = B(p, 2^{i})$ scaled down to unit size, $F_i$
  converge in the Gromov-Hausdorff sense to $F = (x_1,x_2,y,z)$, the
  embedding of $X_0$ to $\mathbf{C}^4$. Such a sequence of embeddings
  can be obtained using adapted sequences of bases for holomorphic
  functions (see \cite[Proposition~3.26]{DS17}). By comparing
  dimensions of the corresponding spaces of holomorphic functions with
  polynomial growth, we see that $x_1^i, x_2^i, y^i, z^i$ must satisfy
  a polynomial equation, each term of which has weight at most
  $6$. For notational simplicity we suppress the index $i$ in the
  discussion below. By making a change of variables that does not
  change the weights of the variables (i.e. completing the squares to
  kill off the terms $x_1,x_2, x_1x_2$ and a shift in $y$ by a scalar
  to kill off the $y^2$ term), the equation reduces to
\begin{align*}
x_1^2+x_2^2+y^3+f(z)y+g(z) = 0,
\end{align*}
where $f(z)$ is a polynomial of degree at most $4$ and $g(z)$ is a
polynomial of degree at most $6$. We claim that $f(z)$ can only be a
constant and that $g(z$) can only be linear. Otherwise, by putting
suitable weights to $x_1,x_2,y,z$, we may assume that the variety
degenerates to one of the following singular hypersurfaces:
\begin{itemize}
\item $x_1^2+x_2^2+y^3+z^ky=0$,
\item $x_1^2+x_2^2+y^3+z^l=0$,
\item $x_1^2+x_2^2+y^3+az^2y+bz^3=0, a,b \ne 0$,
\end{itemize}
where $1 \le k \le 4$ and $2 \le l \le 6$ ($l = 1$ is biholomorphic to
$\mathbf{C}^3$).

In the first two cases, the Milnor number of each isolated singularity
is positive. By Milnor's fibration theorem \cite{Milnor}, the
smoothing has nontrivial topology. In fact, it is homotopy equivalent
to a bouquet of spheres, where the number of spheres is given by the
Milnor number. Therefore it cannot be homeomorphic to
$\mathbf{C}^3$. In the third case, if $27b^2 + 4a^3 \ne 0$ then we
again have an isolated singularity (it is the three-dimensional
$A_2$). If $27b^2 + 4a^3 = 0$, then we have an isolated line
singularity of the form $x_1^2 + x_2^2 + vw^2 = 0$ after a change of
variables. In this case the Milnor fiber is still homotopy equivalent
to a bouquet of spheres \cite{Siersma}. It follows that $f(z)$ can
only be a constant, and $g(z)$ must be linear. For now we conclude
that the image of $F_i$ in $\mathbf{C}^4$ is given by
\[
e_i + a_i z + b_i y + x_1^2 + x_2^2 + y^3 = 0,
\]
where $e_i, a_i$ and $b_i$ are complex numbers. To kill off the
constant term, we make a change of variables $z \to z +
a_i^{-1}e_i$. We need to ensure that that $a_i^{-1}e_i \to 0$ as
$i \to 0$. As pointed out in \cite[Lemma 5]{Sz20}, if we have two sets
of of holomorphic functions $(z, x_1,x_2,y)$ and
$(\tilde{z}, \tilde{x}_1, \tilde{x}_2, \tilde{y})$ on $X$ with weights
$(1,3,3,2)$ such that
\[
az + by + x_1^2 + x_2^2 + y^3 = 0
\]
and
\[
c\tilde{z} + d\tilde{y} + \tilde{x}_1^2 + \tilde{x}_2^2 + \tilde{y}^3 =
0,
\]
then using the fact that both sets of holomorphic functions generate
the space of holomorphic functions with growth rates $\le 6$, we see
that $\tilde{z} = k_1 z$, $\tilde{y} = k_2 y$ and
$(\tilde{x}_1,\tilde{x}_2) = A(x_1,x_2)$ for some scalars $k_1, k_2$
and an invertible matrix $A$ with $A^TA = k_3^2\mathrm{Id}$. We may
assume $A = k_3 \mathrm{Id}$ for some scalar $k_3$. From this it
follows that the two sets of holomorphic functions have a common zero
$o \in X$. Since $F_i$ converges to the standard embedding
$F: B(0,1) \to \mathbf{C}^4$ of the cone, it follows that
$F_i(o) \to 0 \in \mathbf{C}^4$. This implies that
$a_i^{-1}e_i \to 0$. Thus we can absorb this small constant term to
$z$. We make a stop and conclude what we got so far:
\begin{itemize}
\item A sequence of embeddings $F_i = (x_1^i,x_2^i,y^i,z^i)$ of $X$
into $\mathbf{C}^4$ such that $F_i$ converges to $F$ over
$B(p_i,1) \to B(0,1)$.
\item The image $F_i(X)$ is given by the equation
\[
a_i z + b_i y+ x_1^2 + x_2^2 + y^3 = 0,
\]
with $a_i, b_i \to 0$ as $i \to 0$.
\item There exists $o \in X$ such that $F_i(o) = 0$.
\end{itemize}
We still need to conclude (4) and (5) in the statement of the
proposition. Pulling back the volume form using $F_i$, we have
\[
F_i^*(\sqrt{-1}\Omega\wedge\overline\Omega) = |g_i|^2 \omega^3
\]
for some nowhere vanishing polynomial growth holomorphic function
$g_i$ on $X$. Therefore $g_i$ must be a constant (recall that $X$ is
biholomorphic to $\mathbf{C}^3$). By Colding's volume convergence,
\[
2^{-6i} \int_{B(p,2^{2i})} \omega^3 \to \int_{B(0,1)} F^*(\sqrt{-1}\Omega
\wedge \overline\Omega),
\]
it follows that $2^{6i}|g_i|^2 \to 1$ as $i\to 0$. Scaling $z$ by a
factor $\Psi(i^{-1})$-close to $1$, we may assume $|g_i|^2 =
2^{-6i}$. Finally, the image of $F_i$ and $F_{i+1}$ are given by
\[
a_iz + b_iy + x_1^2 + x_2^2 + y^3 = 0
\]
and
\[
a_{i+1}z + b_{i+1}y + x_1^2 + x_2^2 + y^3 = 0,
\]
respectively.  Using the argument finding $o$ such that $F_i(o) = 0$
above, we see that the coefficients of these equations satisfy
\[
\frac{a_i}{k_1a_{i+1}}= \frac{b_i}{k_2b_{i+1}} = \frac{1}{k_2^3} = \frac{1}{k_3^2},
\]
where $k_i$ are such that
$z^{i+1} = k_1z^i, y^{i+1} = k_2 y^i, (x_1^{i+1}, x_2^{i+1}) =
k_3(x_1^i, x_2^i)$. By the definition of adapted sequences of bases
(\cite[Proposition~3.26]{DS17}), we have
$k_1 \to 2^{-1}, k_2 \to 2^{-2}, k_3 \to 2^{-3}$ as $i \to
\infty$. The negative of the powers of $2$ here are the respective
growth rates of the functions. From these, along with the relation
given by the volume forms
\[
2^{6i}F_i^*(\sqrt{-1}\Omega\wedge\overline\Omega) =
2^{6(i+1)}F_{i+1}^*(\sqrt{-1}\Omega\wedge\overline\Omega),
\]
we deduce the limits of $a_i/a_{i+1}$ and $b_i/b_{i+1}$. The same
method shows that $b$ is independent of the sequence constructed
here. Finally, to make $a_i > 0$ we simply rotate the $z$ variable. To
make $b_i \ge 0$, we compose $F_i$ with the following linear
automorphism of $\mathbf{C}^4$:
\[
G_i(z,x_1,x_2,y) = (e^{t_i/2}z, e^{t_i/4}x_1, e^{t_i/4}x_2, e^{t_i/6}y)
\]
(this is $\Phi_t$ in Example~\ref{exmp:cxa2}) for some suitable
$e^{t_i} \in S^1$. Note that $G_i$ preserves the volume form.
\end{proof}

From the proposition we immediately have the following:

\begin{cor}\label{cor:distinguish}
  Suppose $\omega,\omega'$ are two isometric Calabi-Yau metrics on
  $\mathbf{C}^3$ with tangent cone $\mathbf{C} \times A_2$ at
  infinity. Then $b = b'$ in Proposition~\ref{specialembeddings}.
\end{cor}

Note that Corollary~\ref{cor:distinguish} does not imply that the
metrics that we constructed in Theorem~\ref{thm:construction} are
distinct in our sense. Actually, if we apply an automorphism and also
a scaling to a metric in Corollary~\ref{cor:distinguish}, then its
invariant $b$ scales correspondingly.

When $b=0$, we have the following uniqueness result:

\begin{prop}\label{prop:x1unique}
  Let $X$ be a Calabi-Yau manifold biholomorphic to $\mathbf{C}^3$
  with tangent cone $\mathbf{C} \times A_2$ at infinity. If $b = 0$ in
  Proposition~\ref{specialembeddings}, then up to scaling, $X$ is
  isometric to $X_{1,0}$ equipped with the Calabi-Yau metric
  $\omega_{1,0}$ in Theorem~\ref{thm:construction}.
\end{prop}

\begin{proof}
  The proof is very similar to the $\mathbf{C} \times A_1$ case in
  \cite{Sz20}, modulo the special embeddings established in
  Proposition~\ref{specialembeddings} and the computations of
  quadratic harmonic functions and the corresponding vector fields and
  automorphisms on $\mathbf{C} \times A_2$ that are supplemented in
  Example~\ref{exmp:cxa2}. Note that the key reason we can follow the
  proof in \cite{Sz20} is that all the vector fields and automorphisms
  associated to quadratic harmonic functions of
  $\mathbf{C} \times A_2$ actually fix $X_{1,0}$.
\end{proof}

We now turn to distinguishing the metrics in
Theorem~\ref{thm:construction}. For this, we need the following
explicit asymptotic information of the metrics that we have
constructed:

\begin{prop}
\label{prop:asymptotics}
Let $\omega_{1,b}$ be the Calabi-Yau metric on $X_{1,b}$ constructed
in Theorem~\ref{thm:construction}, and let $d$ be the distance
function with respect to $\omega_{1,b}$. Then we have
\[
\lim_{\rho \to \infty} \frac{d(0, (z,x))^2}{|z|^2 + r^2} = 1,
\]
where $0 \in X_{1,b} \subset \mathbf{C}^4$.
\end{prop}

\begin{proof}
Since the metric $\omega_{1,b}$ is a small perturbation in weighted
spaces of the approximate solution $\omega$ on $X_1$, this follows
directly from Corollary~\ref{cor:asymptotics}.
\end{proof}

Together with the above proposition, we can follow the idea of the
proof of Proposition~\ref{specialembeddings} to distinguish the model
metrics $\omega_{1,b}$ on $X_{1,b}$:

\begin{prop}\label{prop:distinguish}
There exist a biholomorphism $F: X_{1,b} \to X_{1,b'}$ and a scaling
$c>0$ such that $F^*\omega_{1,b'} = c^2\omega_{1,b}$ if and only if
$b = b'$.
\end{prop}

\begin{proof}
  Let $(z, x_1,x_2,y)$ and $(z',x_1',x_2',y')$ be the coordinate
  functions on $X_{1,b}$ and $X_{1,b'}$, respectively. Since
  $F^*\omega_{1,b'} = c^2\omega_{1,b}$, The set of functions
  $(z'\circ F,x_1'\circ F,x_2'\circ F,y'\circ F)$ has the same set of
  growth rates that of $(z, x_1,x_2,y)$. By comparing the equations we
  necessarily have $z'\circ F = a_1z$, $y'\circ F = a_2 y$ and
  $x_i'\circ F = a_3 x_i$ (say) for some $a_i \ne 0 \in \mathbf{C}$,
  and
  \[
    \frac{1}{a_1} = \frac{b}{b'a_2} = \frac{1}{a_2^3} = \frac{1}{a_3^2}.
  \]
  In particular we have $F(0) = 0$ and $F(x) \to \infty$ as
  $\rho(x) \to \infty$. By comparing the volume forms we have
  \[
    |a_3|^4 |a_2|^2 = c^6.
  \]
  On the other hand, using the assumption and the fact that $r$ is
  homogeneous, we get
  \[
    F^*\left(\frac{d(0, (z',x'))^2}{|z'|^2 + r^2}\right) = \frac{c^2
      d(0, (z,x))^2}{|a_1|^2(|z|^2 + r^2)}.
  \]
  Taking limit of both sides of the above equation as $\rho \to \infty$
  and using Proposition~\ref{prop:asymptotics}, we conclude that
  $c = |a_1|$. Combining these we see that $c=1$ and $b=b'$.
\end{proof}

\begin{proof}[Proof of Theorem~\ref{thm:A2}]
  This is a combination of Theorem~\ref{thm:construction} and
  Proposition~\ref{prop:distinguish}.
\end{proof}

Based on these elementary observations, we state the following
refinement of a conjecture of Sz\'ekelyhidi~\cite{Sz20}:

\begin{conj}
\label{conj:A2}
The space of Calabi-Yau metrics on $\mathbf{C}^3$ with tangent cone
$\mathbf{C} \times A_2$ at infinity, up to biholomorphism and scaling,
is parametrized by $\mathbf{C}/S^1 \cong \mathbf{R}_{\ge 0}$.
\end{conj}

Difficulties arise when one tries to generalize the decay estimate
approach in \cite{Sz20} to prove Conjecture~\ref{conj:A2}. An initial
technical issue is that the linear automorphisms
$\Phi_t: \mathbf{C}^4 \to \mathbf{C}^4$ in Example~\ref{exmp:cxa2},
which correspond to the quadratic harmonic function $2|z|^2-r^2$ on
the cone $\mathbf{C} \times A_2$, do not preserve the hypersurface
$X_{1,b}$. It is therefore crucial to understand how the metrics
$\omega_{1,b}$, or their potentials $\varphi_b$, change with respect
to the parameter $b$. In the terminology in \cite[Section~3]{CSz}, we
expect that for a given $b$, the metrics
$\Phi_t^* \omega_{1, e^{t/3}b}$ on $X_{1,b}$ for $|t| \ll 1$ form a
family of model metrics parametrized by small quadratic harmonic
functions on the cone $\mathbf{C} \times A_2$. Another difficulty,
which seems more substantial, is due to the parameter space
$[0,\infty)$ being non-compact. In the sequence of special embeddings,
if $a_i$ deviates largely from $2^{-5i}$, then $b_i$ will deviate even
more from $2^{-4i}$ as the decay rate of $b_i$ is slower than
$a_i$. To follow a similar argument as seen in \cite{Sz20}, we need to
have some kind of uniform control of the family of spaces
$(X_{1,b},\omega_{1,b})$ as $b \to \infty$, possibly with suitable
rescalings. To overcome these difficulties, a finer gluing
construction might be needed in order to understand the metric
behavior in the compact region. Alternatively, one could also try to
establish a priori estimates for the complex Monge-Amp\`ere equation
in the maximal volume growth setting. We leave these to future work.

\end{document}